\documentclass[12pt,reqno]{amsart}

\usepackage{amssymb,amsmath}
\usepackage{amsfonts}
\usepackage{amsthm}
\usepackage{mathrsfs}
\usepackage[T1]{fontenc}

\theoremstyle{plain} 
\newtheorem{theorem}{Theorem}
\newtheorem{corollary}[theorem]{Corollary}

\theoremstyle{definition} 

\newtheorem{remark}[theorem]{Remark}

\newcommand{\R}{\ensuremath{\mathbb{R}}}

\newcommand{\T}{\ensuremath{\mathbb{H}}}
\newcommand{\N}{\ensuremath{\mathbb{N}}}
\newcommand{\C}{\ensuremath{\mathbb{C}}}
\newcommand{\I}{\ensuremath{\mathbb{I}}}

\numberwithin{equation}{section}
\numberwithin{theorem}{section}

\small\normalsize

\begin{document}

\title[best Ulam stability constant]{Best constant for Ulam stability of first-order $h$-difference equations with periodic coefficient}
\author[Anderson]{Douglas R. Anderson} 
\address{Department of Mathematics \\
         Concordia College \\
         Moorhead, MN 56562 USA}
\email{andersod@cord.edu}
\author[Onitsuka]{Masakazu Onitsuka}
\address{Okayama University of Science \\
Department of Applied Mathematics \\
Okayama, 700-0005, Japan}
\email{onitsuka@xmath.ous.ac.jp}
\author[Rassias]{John Michael Rassias}
\address{National and Kapodistrian University of Athens \\
Department of Mathematics \& Informatics \\
Attikis 15342, GREECE}
\email{jrassias@primedu.uoa.gr}


\begin{abstract} 
We establish the best (minimum) constant for Ulam stability of first-order linear $h$-difference equations with a periodic coefficient. First, we show Ulam stability and find the Ulam stability constant for a first-order linear equation with a period-two coefficient, and give several examples. In the last section we prove Ulam stability for a periodic coefficient function of arbitrary finite period. Results on the associated first-order perturbed linear equation with periodic coefficient are also included.
\end{abstract}

\maketitle\thispagestyle{empty}


\section{introduction}

Ulam \cite{ulam} introduced a new question of stability, partially answered by Hyers \cite{hyers} and extended by Rassias \cite{rassias}.
In this way Ulam stability, also known as Hyers--Ulam stability or Hyers--Ulam--Rassias stability, has developed in the context of operators, functional equations, differential equations, and difference equations (recurrences); see  Brillou\"{e}t--Belluot, {Brzd\k{e}k}, and Ciepli\'{n}ski \cite{brillouet} for a good broad overview of the literature on this topic, or more recently {Brzd\k{e}k}, Popa, Ra\c{s}a and Xu \cite{brzdek1}. Particular to Ulam stability in the discrete setting, Popa \cite{popa,popa2} had some of the earlier papers, and more recently 
Andr\'{a}s and M\'{e}sz\'{a}ros \cite{andras},
{Brzd\k{e}k} and W\'{o}jcik \cite{brzdek2}, 
Hua, Li and Feng \cite{hua}, 
Jung and Nam \cite{jung}, Nam \cite{nam,nam2,nam3}, Shen \cite{shen}, Rasouli, Abbaszadeh, and Eshaghi \cite{rasouli}, and the present authors \cite{andon,andon2}, have considered recurrences, difference equations, or dynamic equations on time scales in relation to Ulam stability, respectively.

Very little work has been done in the area of Ulam stability and discrete ($h$-difference) equations with periodic coefficients. In what follows we will define what Ulam stability (US) is in the context of first-order $h$-difference equations with a periodic coefficient, and establish parameter values in terms of the periodic coefficient and the constant step size $h>0$ for which the equations exhibit Ulam stability. In the case of Ulam stability, best constants will be found in the sense of the minimum constant needed for an approximate solution (perturbation) to track close to a specific solution of the said equation. We will begin with the easier case of a period-two coefficient, followed by the complete explanation of the general period-$n$ coefficient case. These results set the stage for researchers exploring second and higher order discrete $h$-difference equations with periodicity in the coefficients. 


\section{best constant for first-order equations with two-cycle coefficient}\label{Section2}

Let $h>0$, and define the uniformly discrete set $\T:=\{0,h,2h,3h,\ldots\}$. 

In this section we consider on $\T$ the Ulam stability of the first-order linear homogeneous difference equation with two-cycle (period-two) coefficient
\begin{equation}\label{maineq}
 \Delta_hx(t) - p(t) x(t) = 0, \qquad \Delta_h x(t):=\frac{x(t+h)-x(t)}{h}, 
\end{equation}
where $p:\T\rightarrow\C$ is given by
\begin{equation}\label{two-cycle}
 p(t):=\begin{cases} p_0:& \frac{t}{h}\equiv 0\mod 2, \\ p_1:& \frac{t}{h}\equiv 1\mod 2 \end{cases}
\end{equation}
for $p_0,p_1\in\C\backslash\{\frac{-1}{h}\}$ with $p_0\ne p_1$.
This equation \eqref{maineq} has Ulam stability if and only if there exists a constant $K>0$ with the following property:
\begin{quote}
For arbitrary $\varepsilon>0$, if a function $\phi:\T\rightarrow\C$ satisfies $|\Delta_h\phi(t)-p(t)\phi(t)|\le\varepsilon$ for all $t\in\T$, then there exists a solution $x:\T\rightarrow\C$ of \eqref{maineq} such that $|\phi(t)-x(t)|\le K\varepsilon$ for all $t\in\T$. 
\end{quote}
Such a constant $K$ is called an Ulam stability constant for \eqref{maineq} on $\T$.

The results in this section may be viewed as a discrete version of the results by Fukutaka and Onitsuka \cite{onitsuka1} given for first-order homogeneous linear differential equations with a periodic coefficient, by using a different approach to the proofs and by allowing the periodic coefficient function $p$ in \eqref{maineq} to take non-real (complex) values.


\begin{remark}\label{remark2.1}
Set
\begin{equation}\label{epdef}
 e_p(t):=\prod_{j=0}^{\frac{t-h}{h}}\left(1+hp(jh)\right), \quad\text{where}\quad \prod_{j=0}^{-1}f(j)\equiv 1.
\end{equation}
It is straightforward to check that $e_p$ satisfies \eqref{maineq}, and $e_p(0)=1$.
Let $p_0,p_1\in\C\backslash\{\frac{-1}{h}\}$ with $p_0\ne p_1$. If $|1+hp_0||1+hp_1|=1$, then \eqref{maineq} is not Ulam stable. To see this, let arbitrary $\varepsilon>0$ be given. For $e_p$ given above in \eqref{epdef}, let $\phi$ be defined by
\[ \phi(t):=\varepsilon \ell te_p(t), \qquad t\in\T, \]
where $\ell:=\min\left\{\frac{1}{|1+hp_0|}, 1\right\}$.
Then $\phi$ satisfies the inequality
\[ |\Delta_h\phi(t)-p(t)\phi(t)| = \varepsilon \ell |e_p(t+h)| = \varepsilon \ell\begin{cases} |1+hp_0| &: \frac{t}{h}\equiv 0\mod 2 \\ 1 &: \frac{t}{h}\equiv 1\mod 2 \end{cases}  \le \varepsilon \]
for all $t\in\T$. Since $x(t)=ce_p(t)$ is the general solution of \eqref{maineq}, then
\[ |\phi(t)-x(t)| = \left|\varepsilon\ell t-c\right||e_p(t)|\rightarrow\infty \]
as $t\rightarrow\infty$ for $t\in\T$ and for any $c\in\C$, since $e_p$ is bounded and bounded away from zero; see \cite[Theorem 3.10 (ii)]{andon}. In this case, \eqref{maineq} lacks Ulam stability on $\T$. \hfill$\diamondsuit$
\end{remark}


\begin{theorem}\label{thm22}
Assume $p_0,p_1\in\C\backslash\{\frac{-1}{h}\}$ with $p_0\ne p_1$ and $0 < |1+hp_0||1+hp_1|\ne 1$. 
Let $\varepsilon>0$ be a fixed arbitrary constant, and let the function $\phi:\T\rightarrow\C$ satisfy the inequality
\[ |\Delta_h\phi(t)-p(t)\phi(t)|\le\varepsilon, \qquad t\in\T. \]
Then one of the following holds, where $e_p$ is given in \eqref{epdef}.
\begin{enumerate}
\item If $|1+hp_0||1+hp_1|>1$, then $\displaystyle\lim_{t\rightarrow\infty} \frac{\phi(t)}{e_p(t)}$ exists, and the function $x$ given by 
      \[ x(t):=\left(\lim_{t\rightarrow\infty} \frac{\phi(t)}{e_p(t)}\right)e_p(t) \] 
      is the unique solution of \eqref{maineq} with 
			$$ |\phi(t)-x(t)| \le K_1\varepsilon $$ 
			for all $t\in\T$, where
			\begin{equation}\label{Kmax}
          K_1:=h\max\left\{\frac{1+|1+hp_0|}{-1+|1+hp_0||1+hp_1|}, \frac{1+|1+hp_1|}{-1+|1+hp_0||1+hp_1|}\right\}
      \end{equation}
			is the minimum Ulam stability constant for \eqref{maineq}. 
\item If $0<|1+hp_0||1+hp_1|<1$, then any solution $x$ of \eqref{maineq} with 
      $$|\phi(0)-x(0)| < \varepsilon h\left(\frac{1+|1+hp_1|}{1-|1+hp_0||1+hp_1|}\right)$$ satisfies 
      $$ |\phi(t)-x(t)| < \varepsilon h \max\left\{\frac{1+|1+hp_0|}{1-|1+hp_0||1+hp_1|}, \frac{1+|1+hp_1|}{1-|1+hp_0||1+hp_1|}\right\} $$ 
      for all $t\in\T$. 
\end{enumerate}
\end{theorem}

\begin{proof}
Assume $p_0,p_1\in\C\backslash\{\frac{-1}{h}\}$ with $p_0\ne p_1$ and $0 < |1+hp_0||1+hp_1|\ne 1$. Throughout this proof, as $|\Delta_h\phi(t)-p(t)\phi(t)|\le\varepsilon$ for all $t\in\T$, there exists a function $q:\T\rightarrow\C$ such that
\begin{equation}\label{phiq}
 \Delta_h\phi(t)-p(t)\phi(t)=q(t), \quad |q(t)|\le \varepsilon
\end{equation}
for all $t\in\T$. 

(i): First we consider the case $|1+hp_0||1+hp_1|>1$. The variation of constants formula then yields
\begin{equation}\label{voc}
 \phi(t)=\phi(0)e_p(t) + e_p(t)\sum_{k=0}^{\frac{t-h}{h}}\frac{hq(hk)}{e_p(hk+h)}, \quad\text{with standard assumption}\quad \sum_{k=0}^{-1}f(k) \equiv 0. 
\end{equation}
This $\phi$ can be rewritten as
\begin{equation}\label{vocalt}
 \phi(t)=\left[\phi(0)+\sum_{k=0}^{\infty}\frac{hq(hk)}{e_p(hk+h)}\right]e_p(t) - e_p(t)\sum_{k=\frac{t}{h}}^{\infty}\frac{hq(hk)}{e_p(hk+h)}, 
\end{equation}
where 
\[ x_0:=\phi(0)+\sum_{k=0}^{\infty}\frac{hq(hk)}{e_p(hk+h)}\in\C \]
exists and is finite due to the definition of $e_p$ in \eqref{epdef}, and the assumption $|1+hp_0||1+hp_1|>1$. Clearly 
\[ x(t):=x_0e_p(t), \quad t\in\T \]
is a solution of \eqref{maineq}, and
\[ \lim_{t\rightarrow\infty}\frac{\phi(t)}{e_p(t)}=\phi(0)+\sum_{k=0}^{\infty}\frac{hq(hk)}{e_p(hk+h)} = x_0 \]
exists. Consequently, 
\[ x(t)=\left(\lim_{t\rightarrow\infty}\frac{\phi(t)}{e_p(t)}\right){e_p(t)}, \]
and
\begin{eqnarray*} 
 |\phi(t)-x(t)| &=& \left|-e_p(t)\sum_{k=\frac{t}{h}}^{\infty}\frac{hq(hk)}{e_p(hk+h)}\right| \\
  &\le& h \varepsilon |e_p(t)|\sum_{k=\frac{t}{h}}^{\infty}\frac{1}{|e_p(hk+h)|} \\
	&=& h \varepsilon \left(\frac{1}{|1+hp(t)|}+\frac{1}{|1+hp(t)||1+hp(t+h)|}+\cdots\right) \\
	&=& h \varepsilon \begin{cases} \frac{1+|1+hp_1|}{-1+|1+hp_0||1+hp_1|} &: p(t)=p_0 \\ \frac{1+|1+hp_0|}{-1+|1+hp_0||1+hp_1|} &: p(t)=p_1 \end{cases}
\end{eqnarray*}
holds for all $t\in\T$. Consequently, \eqref{maineq} has Ulam stability with Ulam constant $K_1$ given by \eqref{Kmax}.

We next show in case (i) that $x$ is the unique solution of \eqref{maineq} such that $|\phi(t)-x(t)| \le K_1\varepsilon$ for all $t\in\T$. Suppose $\phi:\T\rightarrow\C$ is an approximate solution of \eqref{maineq} such that
\[ |\Delta_h \phi(t)-p(t)\phi(t)| \le \varepsilon\; \text{ for all } \; t\in\T \] 
for some $\varepsilon>0$. Suppose further that $x_1,x_2:\T\rightarrow\C$ are two different solutions of \eqref{maineq} such that $|\phi(t)-x_j(t)|\le K_1\varepsilon$ for all $t\in\T$, for $j=1,2$. Then we have for constants $c_j\in\C$ that
\[ x_j(t) = c_je_p(t), \quad c_1\neq c_2,\]
and
\[ |e_p(t)||c_1-c_2| = |x_1(t)-x_2(t)| \le |x_1(t)-\phi(t)| + |\phi(t)-x_2(t)| \le 2K_1\varepsilon; \]
letting $t\rightarrow \infty$ yields $\infty<2K_1\varepsilon$, a contradiction. Consequently, $x$ is the unique solution of \eqref{maineq} such that $|\phi(t)-x(t)| \le K_1\varepsilon$ for all $t\in\T$.

Finally we show in case (i) that $K_1$ in \eqref{Kmax} is the minimum Ulam constant. In \eqref{phiq}, if $q(t)\equiv\varepsilon$ for $t\in\T$, then \eqref{voc} and \eqref{vocalt} imply the function $\phi:\T\rightarrow\C$ given by 
\[ \phi(t):=\left[\phi(0)+\sum_{k=0}^{\infty}\frac{h\varepsilon}{e_p(hk+h)}\right]e_p(t) - e_p(t)\sum_{k=\frac{t}{h}}^{\infty}\frac{h\varepsilon}{e_p(hk+h)} \]
satisfies the equality
\[ |\Delta_h\phi(t)-p(t)\phi(t)| = \varepsilon, \qquad t\in\T. \]
As 
\[ x(t):=\left[\phi(0)+\sum_{k=0}^{\infty}\frac{h\varepsilon}{e_p(hk+h)}\right]e_p(t) \]
is a solution of \eqref{maineq},
\begin{eqnarray*} 
 |\phi(t)-x(t)| &=& h \varepsilon |e_p(t)|\sum_{k=\frac{t}{h}}^{\infty}\frac{1}{|e_p(hk+h)|} = h \varepsilon \begin{cases} \frac{1+|1+hp_1|}{-1+|1+hp_0||1+hp_1|} &: p(t)=p_0 \\ \frac{1+|1+hp_0|}{-1+|1+hp_0||1+hp_1|} &: p(t)=p_1 \end{cases} \\
&\le& K_1\varepsilon
\end{eqnarray*}
holds for all $t\in\T$. As a result, all parts of (i) hold.

(ii): Now assume $0<|1+hp_0||1+hp_1|<1$. It is straightforward to check that $\phi$ takes the form
\[ \phi(t) = \phi(0)e_p(t) + e_p(t)\sum_{k=0}^{\frac{t-h}{h}}\frac{hq(kh)}{e_p(kh+h)} \]
by the variation of constants formula. Let $x$ be any solution of \eqref{maineq} with
\[ |\phi(0)-x(0)| < \varepsilon h\left(\frac{1+|1+hp_1|}{1-|1+hp_0||1+hp_1|}\right). \]
Then $x$ takes the form
\[ x(t)=x(0)e_p(t), \quad t\in\T, \]
and we have
\begin{eqnarray*} 
 \phi(t)-x(t) &=& e_p(t) \left(\phi(0)-x(0)\right)+ e_p(t) \sum_{k=0}^{\frac{t-h}{h}}\frac{hq(kh)}{e_p(kh+h)}.
\end{eqnarray*}
It follows that
\begin{eqnarray*}
 |\phi(t)-x(t)| &\le& |e_p(t)||\phi(0)-x(0)| + \varepsilon h|e_p(t)|\sum_{k=0}^{\frac{t-h}{h}}\frac{1}{|e_p(kh+h)|} \\
 &<& |e_p(t)|\varepsilon h\left(\frac{1+|1+hp_1|}{1-|1+hp_0||1+hp_1|}\right) + \varepsilon h|e_p(t)|\sum_{k=0}^{\frac{t-h}{h}}\frac{1}{|e_p(kh+h)|}.
\end{eqnarray*}
Now 
\[ |e_p(t)| = \begin{cases} |1+hp_0|^{n}|1+hp_1|^{n-1} &: t=(2n-1)h \\ |1+hp_0|^{n}|1+hp_1|^{n} &: t=2nh \end{cases} \]
and 
\begin{eqnarray*}
 |e_p(t)|\sum_{k=0}^{\frac{t-h}{h}}\frac{1}{|e_p(kh+h)|} 
 &=& \begin{cases} \left(\frac{1-|1+hp_0|^n|1+hp_1|^n+|1+hp_0|-|1+hp_0|^{n}|1+hp_1|^{n-1}}{1-|1+hp_0||1+hp_1|}\right) &: t=(2n-1)h \\  
 \left(\frac{1-|1+hp_0|^n|1+hp_1|^n+|1+hp_1|-|1+hp_0|^{n}|1+hp_1|^{n+1}}{1-|1+hp_0||1+hp_1|}\right) &: t=2nh. \end{cases}
\end{eqnarray*}
Piecing it all together,
\begin{eqnarray*}
 |\phi(t)-x(t)| &<& \varepsilon h  \begin{cases} \frac{|1+hp_0|^{n}|1+hp_1|^{n-1}+|1+hp_0|^{n}|1+hp_1|^{n}+1-|1+hp_0|^n|1+hp_1|^n+|1+hp_0|-|1+hp_0|^{n}|1+hp_1|^{n-1}}{1-|1+hp_0||1+hp_1|} \\ \frac{|1+hp_0|^{n}|1+hp_1|^{n}+|1+hp_0|^{n}|1+hp_1|^{n+1}+1-|1+hp_0|^n|1+hp_1|^n+|1+hp_1|-|1+hp_0|^{n}|1+hp_1|^{n+1}}{1-|1+hp_0||1+hp_1|} \end{cases} \\
&=& \varepsilon h  \begin{cases} \frac{1+|1+hp_0|}{1-|1+hp_0||1+hp_1|} &: \frac{t}{h}\equiv 1\mod 2 \\ \frac{1+|1+hp_1|}{1-|1+hp_0||1+hp_1|} &: \frac{t}{h}\equiv 0\mod 2 \end{cases} \\
&\le& \varepsilon h\max\left\{ \frac{1+|1+hp_0|}{1-|1+hp_0||1+hp_1|}, \frac{1+|1+hp_1|}{1-|1+hp_0||1+hp_1|}\right\}
\end{eqnarray*}
for all $t\in\T$. Thus, (ii) holds and the proof is complete.
\end{proof}

Using Theorem \ref{thm22}, we get the following result immediately.


\begin{theorem}\label{thm23}
Assume $p_0,p_1\in\C\backslash\{\frac{-1}{h}\}$ with $p_0\ne p_1$ and $0 < |1+hp_0||1+hp_1|\ne 1$. Then \eqref{maineq} has Ulam stability with Ulam stability constant 
\begin{equation}\label{Kmax0}
   K_0:=h\max\left\{\frac{1+|1+hp_0|}{|1-|1+hp_0||1+hp_1||}, \frac{1+|1+hp_1|}{|1-|1+hp_0||1+hp_1||}\right\}
\end{equation}
 on $\T$. Moreover, if $|1+hp_0||1+hp_1|>1$, then $K_0$ is the minimum Ulam stability constant for \eqref{maineq}. 
\end{theorem}


\begin{remark}\label{remark2.3}
It is known that the constant coefficient $h$-difference equation
\[ \Delta_hx(t) - a x(t) = 0 \]
lacks Ulam stability on $\T$ when $a = 0$ or $-2/h$; see \cite[Remark 3.3]{andon} and \cite[Remark 1.1]{onitsuka2}. When $p(t)$ has infinitely many zeros or infinitely many points satisfying $p(t)=-2/h$, does \eqref{maineq} have Ulam stability on $\T$? Our result can give a positive answer to this question. Specifically, consider the functions
\[ p_0(t):=\begin{cases} p_0 \in \R\setminus \{0,-\frac{2}{h}\}:& \frac{t}{h}\equiv 0\mod 2, \\ p_1=0:& \frac{t}{h}\equiv 1\mod 2 \end{cases} \]
and
\[ p_{-\frac{2}{h}}(t):=\begin{cases} p_0 \in \R\setminus \{0,-\frac{2}{h}\}:& \frac{t}{h}\equiv 0\mod 2, \\ p_1=-\frac{2}{h}:& \frac{t}{h}\equiv 1\mod 2 \end{cases} \]
for $t\in\T$; that is, $p_0(t)$ has infinitely many zeros and $p_{-\frac{2}{h}}(t)$ has infinitely many points satisfying $p(t)=-2/h$. Clearly, we see that $p_0\ne p_1$ and $0\ne |1+hp_0||1+hp_1|\ne 1$ hold. The following results are obtained by Theorem \ref{thm23} and simple calculations. Our main equation with $p(t)=p_0(t)$ or $p_{-\frac{2}{h}}(t)$ has Ulam stability with Ulam constant $K_0$ given by \eqref{Kmax0}. In these cases, Ulam constants $K_0$ are represented in the same form
\[ K_0 = \begin{cases} \frac{h^2p_0}{2+hp_0}:& p_0<-\frac{2}{h}, \\
 \frac{2p_0}{2+hp_0}:& -\frac{2}{h}<p_0<-\frac{1}{h}, \\
 \frac{-2}{p_0}:& -\frac{1}{h}<p_0<0, \\
 \frac{2+hp_0}{p_0}:& 0<p_0. \\
 \end{cases}\]
Additionally, $\frac{h^2p_0}{2+hp_0}$ and $\frac{2+hp_0}{p_0}$ are the best (minimum) constants for Ulam stability when $p_0<-\frac{2}{h}$ and $0<p_0$, respectively.
\hfill$\diamondsuit$
\end{remark}


\begin{remark}\label{hilgercircle}
Extending Remark \ref{remark2.3}, given the two-cycle $p_0,p_1\in\C\backslash\{\frac{-1}{h}\}$ with $p_0\ne p_1$, and fixed step size $h>0$, the key quantity is $|1+hp_0| |1+hp_1|$. 
Let $\I_h$ be the Hilger imaginary circle \cite{hilger} or \cite[pages 51--53]{bp}. If both $p_0,p_1\in\I_h$, that is to say if there exist $\alpha_j,\beta_j\in\R$ with
\[ p_j=\alpha_j+i\beta_j, \qquad \left(\alpha_j+\frac{1}{h}\right)^2+\beta_j^2=\frac{1}{h^2}, \quad j\in\{0,1\}, \] 
then $|1+hp_0| |1+hp_1|=1$ and \eqref{maineq} is not Ulam stable by Remark \ref{remark2.1}. Now for $p_0,p_1\in\C\backslash\{\frac{-1}{h}\}$, consider the two-cycle coefficient function
\[ p(t):=\begin{cases} p_0 \in\C\setminus\I_h &: \frac{t}{h}\equiv 0\mod 2, \\ p_1\in\I_h &: \frac{t}{h}\equiv 1\mod 2 \end{cases} \]
for $t\in\T$; here $p$ lands on $\I_h$ infinitely often. Clearly $p_0\ne p_1$, $|1+hp_1|=1$, and $0 < |1+hp_0||1+hp_1|\ne 1$ hold. Again by Theorem \ref{thm23}, equation \eqref{maineq} has Ulam stability with Ulam constant $K_0$ given by \eqref{Kmax0}. In this case, the Ulam constant $K_0$ is
\[ K_0 = h\begin{cases} 
 \frac{2}{1-|1+hp_0|} &: 0<|1+hp_0|<1, \\
 \frac{|1+hp_0|+1}{|1+hp_0|-1} &: |1+hp_0|>1.
 \end{cases} \]
Additionally, $K_0=\frac{h|1+hp_0|+h}{|1+hp_0|-1}$ is the best (minimum) constant for Ulam stability when $|1+hp_0|>1$, that is, when $p_0$ is outside the Hilger imaginary circle and $p_1$ is on it.
\hfill$\diamondsuit$
\end{remark}


\section{perturbed linear equations}

In this section, we consider the first-order perturbed linear equation
\begin{equation}\label{perturbed}
 \Delta_h\phi(t) - p(t) \phi(t) = f(t,\phi(t)),
\end{equation}
where $p(t)$ is given in \eqref{two-cycle} and $f(t,\phi)$ is a complex-valued function on $\T \times \C$. We say that the solutions of \eqref{perturbed} are uniform-ultimately bounded for a bound $B$ if and only if there exists a constant $B>0$ with the following property:
\begin{quote}
For any $\alpha>0$, there exists a $T(\alpha)>0$ such that $|\phi_0|<\alpha$ with $\phi_0 \in \C$ imply that $|\phi(t)|<B$ for all $t\ge T(\alpha)$ with $t\in \T$, where $\phi(t)$ is a solution of \eqref{perturbed} satisfying $\phi(0)=\phi_0$. 
\end{quote}
The uniform-ultimate boundedness of the solutions has long been treated as an important problem in the field of ordinary differential equations and dynamical systems. For example, see \cite{Michel,Yoshi}. Using Theorem \ref{thm22}, we can obtain the following result.

\begin{corollary}\label{boundedness}
Let $\delta>0$ be an arbitrary constant. Suppose that there exists an $L>0$ such that $|f(t,\phi)|\le L$ for all $(t,\phi)\in \T \times \C$. Suppose also that all solutions of \eqref{perturbed} exist on $\T$. If $0<|1+hp_0||1+hp_1|<1$, then all solutions of \eqref{perturbed} are uniform-ultimately bounded for a bound $LK_0+\delta$, where $K_0$ is given in \eqref{Kmax0}.
\end{corollary}

\begin{proof}
Set $B = LK_0+\delta$ for fixed $\delta>0$. Let $\phi(t)$ be the solution of \eqref{perturbed} with the initial condition $\phi(0)=\phi_0 \in \C$ with $|\phi_0|<\alpha$, where $\alpha$ is a fixed arbitrary constant. Since
\[ |\Delta_h\phi(t) - p(t) \phi(t)| = |f(t,\phi(t))| \le L \]
holds for all $t\in\T$, from Theorem \ref{thm22} (ii) we can find a solution $x$ of \eqref{maineq} with the initial condition
\[ |\phi_0-x(0)| < L h\left(\frac{1+|1+hp_1|}{1-|1+hp_0||1+hp_1|}\right) \le LK_0 \]
that satisfies
\[ |\phi(t)-x(t)| < LK_0 \]
for all $t\in\T$. This solution is written as $x(t) = x(0)e_p(t)$ on $\T$. Consequently, we have
\begin{eqnarray*}
  |\phi(t)| &\le& |\phi(t)-x(t)|+|x(t)| \\
  &<& LK_0+|x(t)| = LK_0+|x(0)e_p(t)| \\
  &\le& LK_0+(|x(0)-\phi_0|+|\phi_0|)|e_p(t)| \\
  &<& LK_0+(LK_0+\alpha)|e_p(t)|
\end{eqnarray*}
for all $t\in\T$. Note here that $0< |e_p(t)| \le \max\{1,|1+hp_0|\}$ holds for all $t\in\T$ by $0<|1+hp_0||1+hp_1|<1$. Hence, together with this and the above inequality, we obtain
\[ |\phi(t)| < LK_0+(LK_0+\alpha)\max\{1,|1+hp_0|\} \]
for all $t\in\T$. If $(LK_0+\alpha)\max\{1,|1+hp_0|\} \le \delta$ then $|\phi(t)| < B$ for all $t\in\T$; that is, $\phi(t)$ is uniform-ultimately bounded for a bound $B$. Next, we will consider the case $(LK_0+\alpha)\max\{1,|1+hp_0|\} > \delta$. Set
\[ T(\alpha) = h\left(2\log_{|1+hp_0||1+hp_1|}\frac{\delta}{(LK_0+\alpha)\max\{1,|1+hp_0|\}}+1\right). \]
If $t=(2n-1)h \ge T(\alpha)$ then
\begin{eqnarray*}
  |e_p(t)| &\le& \max\{1,|1+hp_0|\}(|1+hp_0||1+hp_1|)^{n-1} = \max\{1,|1+hp_0|\}(|1+hp_0||1+hp_1|)^{\frac{1}{2}\left(\frac{t}{h}-1\right)}\\
  &\le& \max\{1,|1+hp_0|\}(|1+hp_0||1+hp_1|)^{\frac{1}{2}\left(\frac{T(\alpha)}{h}-1\right)} = \frac{\delta}{LK_0+\alpha},
\end{eqnarray*}
and if $t=2nh \ge T(\alpha)$ then
\begin{eqnarray*}
  |e_p(t)| &\le& (|1+hp_0||1+hp_1|)^n = (|1+hp_0||1+hp_1|)^{\frac{t}{2h}}\\
  &\le& (|1+hp_0||1+hp_1|)^{\frac{T(\alpha)}{2h}}\\
  &<& \frac{\delta}{(LK_0+\alpha)\max\{1,|1+hp_0|\}} \le \frac{\delta}{LK_0+\alpha}.
\end{eqnarray*}
Consequently, we have
\[ |\phi(t)| < LK_0+\delta = B \]
for all $t\in\T$. Thus, the proof is now complete.
\end{proof}


\section{best constant for first-order equations with $n$-cycle coefficient}

In this section we consider on $\T$ the general extension of Section \ref{Section2} to arbitrary finite period, namely the Ulam stability of the first-order linear homogeneous difference equation with $n$-cycle (period $n$) coefficient
\begin{equation}\label{neq}
 \Delta_hx(t) - p(t) x(t) = 0, \qquad \Delta_h x(t):=\frac{x(t+h)-x(t)}{h}, 
\end{equation}
where $n\in\N$, $p:\T\rightarrow\C$ is given by
\begin{equation}\label{n-cycle}
 p(t):= p_k \quad\text{if}\quad \frac{t}{h}\equiv k\mod n
\end{equation}
for $k\in\{0,1,\ldots,n-1\}$, and $p_0,p_1,\ldots,p_{n-1}\in\C\backslash\{\frac{-1}{h}\}$ such that the coefficient function $p$ is periodic with period $n$, and $p$ is not periodic for any $k<n$.


\begin{remark}\label{remark4.1}
It is straightforward to check that $e_p$ in \eqref{epdef} satisfies \eqref{neq}, and $e_p(0)=1$.
Let $p_0,p_1,\ldots,p_{n-1}\in\C\backslash\{\frac{-1}{h}\}$. 
For convenience, note that
\[ |e_p(kh)|=|1+hp_0| |1+hp_1| \cdots |1+hp_{k-1}|, \quad k\in\{1,2,\ldots,n\}. \]
If $|e_p(nh)|=1$, then \eqref{neq} is not Ulam stable. To see this, let arbitrary $\varepsilon>0$ be given. For $e_p$ given above in \eqref{epdef}, let $\phi$ be defined by
\[ \phi(t):=\varepsilon \ell te_p(t), \qquad t\in\T, \]
where $\ell:=\min\left\{\frac{1}{|e_p(h)|}, \frac{1}{|e_p(2h)|}, \ldots, \frac{1}{|e_p((n-1)h)|}, \frac{1}{|e_p(nh)|}\right\}$.
Then $\phi$ satisfies the inequality
\[ |\Delta_h\phi(t)-p(t)\phi(t)| = \varepsilon \ell |e_p(t+h)| = \varepsilon \ell |e_p((k+1)h)|  \le \varepsilon \]
for $\frac{t}{h}\equiv k\mod n$ for all $t\in\T$. Since $x(t)=ce_p(t)$ is the general solution of \eqref{maineq}, then
\[ |\phi(t)-x(t)| = \left|\varepsilon\ell t-c\right||e_p(t)|\rightarrow\infty \]
as $t\rightarrow\infty$ for $t\in\T$ and for any $c\in\C$, since $e_p$ is bounded and bounded away from zero; see \cite[Theorem 3.10 (ii)]{andon}. In this case, \eqref{neq} lacks Ulam stability on $\T$. \hfill$\diamondsuit$
\end{remark}


\begin{remark}
Assume the coefficient function $p$ satisfies \eqref{n-cycle} for $p_0,p_1,\ldots,p_{n-1}\in\C\backslash\{\frac{-1}{h}\}$. Let 
\begin{equation}\label{sum0}
 S_0 =\frac{1}{|1+hp_0|} + \frac{1}{|1+hp_0||1+hp_{1}|} + \cdots + \frac{1}{|1+hp_0||1+hp_1|\cdots|1+hp_{n-1}|} 
\end{equation}
and
\begin{eqnarray}
	S_k &=&\frac{1}{|1+hp_k|} + \frac{1}{|1+hp_k||1+hp_{k+1}|} + \cdots \nonumber \\
	& &  + \frac{1}{|1+hp_k|\cdots|1+hp_{n-1}|} +  \frac{1}{|1+hp_k|\cdots|1+hp_{n-1}||1+hp_0|} + \cdots \nonumber \\
	& & + \frac{1}{|1+hp_k|\cdots|1+hp_{n-1}||1+hp_0|\cdots|1+hp_{k-1}|} \label{sumk}
\end{eqnarray}
for $k\in\{1,\ldots,n-1\}$. We will refer to these sums in the following theorem. \hfill$\diamondsuit$
\end{remark} 


\begin{theorem}\label{thm43}
Assume the coefficient function $p$ satisfies \eqref{n-cycle} for $p_0,p_1,\ldots,p_{n-1}\in\C\backslash\{\frac{-1}{h}\}$, with 
$0 < |e_p(nh)| \ne 1$. Let $\varepsilon>0$ be a fixed arbitrary constant, and let the function $\phi:\T\rightarrow\C$ satisfy the inequality
\[ |\Delta_h\phi(t)-p(t)\phi(t)|\le\varepsilon, \qquad t\in\T. \]
Then one of the following holds, where $e_p$ is given in \eqref{epdef}, and $S_0$, $S_k$ are given in \eqref{sum0}, \eqref{sumk}, respectively.
\begin{enumerate}
\item If $|e_p(nh)|>1$, then $\displaystyle\lim_{t\rightarrow\infty} \frac{\phi(t)}{e_p(t)}$ exists, and the function $x$ given by 
      \[ x(t):=\left(\lim_{t\rightarrow\infty} \frac{\phi(t)}{e_p(t)}\right)e_p(t) \] 
      is the unique solution of \eqref{neq} with 
			$$ |\phi(t)-x(t)| \le K_n\varepsilon $$ 
			for all $t\in\T$, where
			\begin{eqnarray}\label{nKmax}
        K_n:=\frac{h |e_p(nh)|}{-1+|e_p(nh)|}\max\left\{S_0, S_1, \ldots, S_{n-1}\right\}
      \end{eqnarray}
			is the minimum Ulam stability constant for \eqref{neq}, using \eqref{sum0} and \eqref{sumk}. 
\item If $0<|e_p(nh)|<1$, then any solution $x$ of \eqref{neq} with 
      $$|\phi(0)-x(0)| < \varepsilon h\left(\frac{|e_p(nh)|S_0}{1-|e_p(nh)|}\right)$$ satisfies 
      $$ |\phi(t)-x(t)| < \frac{\varepsilon h|e_p(nh)|}{1-|e_p(nh)|} \max\left\{S_0, S_1, \ldots, S_{n-1}\right\} $$ 
      for all $t\in\T$. 
\end{enumerate}
\end{theorem}

\begin{proof}
Assume the coefficient function $p$ satisfies \eqref{n-cycle} such that $p_0,p_1,\ldots,p_{n-1}\in\C\backslash\{\frac{-1}{h}\}$, with $0 < |e_p(nh)|\ne 1$. Throughout this proof, as 
$|\Delta_h\phi(t)-p(t)\phi(t)|\le\varepsilon$ for all $t\in\T$, there exists a function $q:\T\rightarrow\C$ such that
\begin{equation}\label{nphiq}
 \Delta_h\phi(t)-p(t)\phi(t)=q(t), \quad |q(t)|\le \varepsilon
\end{equation}
for all $t\in\T$. 

(i): First we consider the case $ |e_p(nh)| > 1$. The variation of constants formula again yields \eqref{voc}.
This $\phi$ can be rewritten as \eqref{vocalt}. As in the proof of Theorem \ref{thm22},
\[ x(t)=\left(\lim_{t\rightarrow\infty}\frac{\phi(t)}{e_p(t)}\right){e_p(t)}, \]
and for $k\in\{1,2,\ldots,n-1\}$ we see that
\begin{eqnarray*} 
 |\phi(t)-x(t)|
  & = & \left|-e_p(t)\sum_{j=\frac{t}{h}}^{\infty}\frac{hq(hj)}{e_p(hj+h)}\right| \\
  &\le& h \varepsilon |e_p(t)|\sum_{j=\frac{t}{h}}^{\infty}\frac{1}{|e_p(hj+h)|}  \\
  & = & h \varepsilon \left(\frac{1}{|1+hp(t)|}+\frac{1}{|1+hp(t)||1+hp(t+h)|}+\cdots\right) \\
	& = & h\varepsilon\left(\sum_{j=0}^{\infty}\frac{1}{|e_p(nh)|^j}\right) \begin{cases} S_0 &: p(t)=p_0, \\ S_k &: p(t)=p_k \end{cases}\\
	& = & \frac{h\varepsilon}{1-\frac{1}{|e_p(nh)|}} \begin{cases} S_0 &: p(t)=p_0, \\ S_k &: p(t)=p_k \end{cases}
\end{eqnarray*}
holds for all $t\in\T$, where $S_0$ and $S_k$ are given in \eqref{sum0} and \eqref{sumk}, respectively. 
Consequently,  
$$ |\phi(t)-x(t)| \le \frac{h\varepsilon |e_p(nh)|}{-1+|e_p(nh)|}\max\{S_0,S_k\}, \quad k\in\{1,\ldots,n-1\}, $$
and \eqref{neq} has Ulam stability with Ulam constant $K_n$ given by \eqref{nKmax}.

We next show in case (i) that $x$ is the unique solution of \eqref{neq} such that $|\phi(t)-x(t)| \le K_n\varepsilon$ for all $t\in\T$. Suppose $\phi:\T\rightarrow\C$ is an approximate solution of \eqref{neq} such that
\[ |\Delta_h \phi(t)-p(t)\phi(t)| \le \varepsilon\; \text{ for all } \; t\in\T \] 
for some $\varepsilon>0$. Suppose further that $x_1,x_2:\T\rightarrow\C$ are two different solutions of \eqref{neq} such that $|\phi(t)-x_j(t)|\le K_n\varepsilon$ for all $t\in\T$, for $j=1,2$. Then we have for constants $c_j\in\C$ that
\[ x_j(t) = c_je_p(t), \quad c_1\neq c_2,\]
and
\[ |e_p(t)||c_1-c_2| = |x_1(t)-x_2(t)| \le |x_1(t)-\phi(t)| + |\phi(t)-x_2(t)| \le 2K_n\varepsilon; \]
letting $t\rightarrow \infty$ yields $\infty<2K_n\varepsilon$, a contradiction. Consequently, $x$ is the unique solution of \eqref{neq} such that $|\phi(t)-x(t)| \le K_n\varepsilon$ for all $t\in\T$.

Finally we show in case (i) that $K_n$ in \eqref{nKmax} is the minimum Ulam constant. In \eqref{nphiq}, if $q(t)\equiv\varepsilon$ for $t\in\T$, then \eqref{voc} and \eqref{vocalt} imply the function $\phi:\T\rightarrow\C$ given by 
\[ \phi(t):=\left[\phi(0)+\sum_{j=0}^{\infty}\frac{h\varepsilon}{e_p(hj+h)}\right]e_p(t) - e_p(t)\sum_{j=\frac{t}{h}}^{\infty}\frac{h\varepsilon}{e_p(hj+h)} \]
satisfies the equality
\[ |\Delta_h\phi(t)-p(t)\phi(t)| = \varepsilon, \qquad t\in\T. \]
As 
\[ x(t):=\left[\phi(0)+\sum_{j=0}^{\infty}\frac{h\varepsilon}{e_p(hj+h)}\right]e_p(t) \]
is a solution of \eqref{neq},
\begin{eqnarray*} 
 |\phi(t)-x(t)| &=& h \varepsilon |e_p(t)|\sum_{j=\frac{t}{h}}^{\infty}\frac{1}{|e_p(hj+h)|} = \frac{h\varepsilon}{1-\frac{1}{|e_p(nh)|}} \begin{cases} S_0 &: p(t)=p_0, \\ S_k &: p(t)=p_k \end{cases} \\
&\le& K_n\varepsilon
\end{eqnarray*}
holds for all $t\in\T$. As a result, all parts of (i) hold.

(ii): Now assume $0<|e_p(nh)|<1$. It is straightforward to check that $\phi$ takes the form
\[ \phi(t) = \phi(0)e_p(t) + e_p(t)\sum_{j=0}^{\frac{t-h}{h}}\frac{hq(jh)}{e_p(jh+h)} \]
by the variation of constants formula. Let $x$ be any solution of \eqref{maineq} with
\[ |\phi(0)-x(0)| <  \varepsilon h\left(\frac{|e_p(nh)|S_0}{1-|e_p(nh)|}\right), \]
where $S_0$ is as in \eqref{sum0}. Then $x$ takes the form
\[ x(t)=x(0)e_p(t), \quad t\in\T, \]
and we have
\begin{eqnarray*} 
 \phi(t)-x(t) &=& e_p(t) \left(\phi(0)-x(0)\right)+ e_p(t) \sum_{j=0}^{\frac{t-h}{h}}\frac{hq(jh)}{e_p(jh+h)}.
\end{eqnarray*}
It follows that
\begin{eqnarray*}
 |\phi(t)-x(t)| &\le& |e_p(t)||\phi(0)-x(0)| + \varepsilon h|e_p(t)|\sum_{j=0}^{\frac{t-h}{h}}\frac{1}{|e_p(jh+h)|} \\
 &<& \varepsilon h|e_p(t)|\left(\frac{|e_p(nh)|S_0}{1-|e_p(nh)|}\right) + \varepsilon h|e_p(t)|\sum_{j=0}^{\frac{t-h}{h}}\frac{1}{|e_p(jh+h)|}.
\end{eqnarray*}
Now 
\[ |e_p(t)| = \begin{cases} 
 |1+hp_0|^{m}|1+hp_1|^{m-1}\cdots|1+hp_{n-1}|^{m-1} &: t=(mn-n+1)h \\ 
 |1+hp_0|^{m}|1+hp_1|^{m}|1+hp_2|^{m-1}\cdots|1+hp_{n-1}|^{m-1} &: t=(mn-n+2)h \\
 \hspace{1in} \vdots & \hspace{0.5in} \vdots \\
 |1+hp_0|^{m}\cdots|1+hp_{n-3}|^{m}|1+hp_{n-2}|^{m-1}|1+hp_{n-1}|^{m-1} &: t=(mn-2)h \\
 |1+hp_0|^{m}\cdots|1+hp_{n-2}|^{m}|1+hp_{n-1}|^{m-1} &: t=(mn-1)h \\
 |1+hp_0|^{m}|1+hp_1|^{m}\cdots|1+hp_{n-2}|^{m}|1+hp_{n-1}|^{m} &: t=mnh.
\end{cases} \]
For $t=(mn-n+1)h$, we have
\begin{eqnarray*}
 |\phi(t)-x(t)|
 &<& \varepsilon h|e_p(t)|\left(\frac{|e_p(nh)|S_0}{1-|e_p(nh)|}\right) + \varepsilon h|e_p(t)|\sum_{j=0}^{\frac{t-h}{h}}\frac{1}{|e_p(jh+h)|}\\
 &=& \varepsilon h|e_p(t)|\left(\frac{|e_p(nh)|S_0}{1-|e_p(nh)|}\right) + \varepsilon h|e_p(t)|\left[\frac{|e_p(nh)|S_0}{1-|e_p(nh)|}\left(\frac{1}{|e_p(nh)|^{m-1}}-1\right)+\frac{1}{|e_p(t)|} \right]\\
 &=& \frac{\varepsilon h|e_p(nh)|S_1}{1-|e_p(nh)|},
\end{eqnarray*}
and for $t=(mn-n+2)h$, we have
\begin{eqnarray*}
 |\phi(t)-x(t)|
 &<& \varepsilon h|e_p(t)|\left(\frac{|e_p(nh)|S_0}{1-|e_p(nh)|}\right) + \varepsilon h|e_p(t)|\sum_{j=0}^{\frac{t-h}{h}}\frac{1}{|e_p(jh+h)|}\\
 &=& \varepsilon h|e_p(t)|\left(\frac{|e_p(nh)|S_0}{1-|e_p(nh)|}\right)\\
 & & + \varepsilon h|e_p(t)|\left[\frac{|e_p(nh)|S_0}{1-|e_p(nh)|}\left(\frac{1}{|e_p(nh)|^{m-1}}-1\right)+\frac{1}{|e_p(nh)|^{m-1}|1+hp_0|}+\frac{1}{|e_p(t)|} \right]\\
 &=& \frac{\varepsilon h|e_p(nh)|S_2}{1-|e_p(nh)|};
\end{eqnarray*}
this pattern continues until for $t=(mn-1)h$ we have
\[ |\phi(t)-x(t)| < \frac{\varepsilon h|e_p(nh)|S_{n-1}}{1-|e_p(nh)|} \]
and for $t=mnh$ we have
\[ |\phi(t)-x(t)| < \frac{\varepsilon h|e_p(nh)|S_{0}}{1-|e_p(nh)|}. \]
Putting it all together,
\begin{eqnarray*}
 |\phi(t)-x(t)|
 & < & \varepsilon h|e_p(t)|\left(\frac{|e_p(nh)|S_0}{1-|e_p(nh)|}\right) 
        + \varepsilon h|e_p(t)|\sum_{j=0}^{\frac{t-h}{h}}\frac{1}{|e_p(jh+h)|}\\
 &\le& \frac{\varepsilon h|e_p(nh)|}{1-|e_p(nh)|}\max\left\{S_0,S_1,\ldots,S_{n-1}\right\}
\end{eqnarray*}
for all $t\in\T$. Thus, (ii) holds and the proof is complete.
\end{proof}

Using Theorem \ref{thm43}, we get the following result immediately.


\begin{theorem}\label{thm44}
Assume the coefficient function $p$ satisfies \eqref{n-cycle} for $p_0,p_1,\ldots,p_{n-1}\in\C\backslash\{\frac{-1}{h}\}$, with 
$0 < |e_p(nh)| \ne 1$. Let $S_0$ and $S_k$ for $k\in\{1,2,\ldots,n-1\}$ be given by \eqref{sum0} and \eqref{sumk}, respectively. 
Then \eqref{neq} has Ulam stability with Ulam stability constant 
\begin{equation}\label{nKmax0}
   K_0:=\frac{h|e_p(nh)|}{\left|1-|e_p(nh)|\right|}\max\left\{S_0,S_1,\ldots,S_{n-1}\right\}
\end{equation}
 on $\T$. Moreover, if $|e_p(nh)|>1$, then $K_0$ is the minimum Ulam stability constant for \eqref{maineq}. 
\end{theorem}


\begin{remark}\label{remark4.5}
Let $n=3$ and $p_0\not=\frac{-1}{h}$, and consider the 3-cycle
\[ p(t):=\begin{cases} p_0 \in \C\setminus\I_h &: \frac{t}{h}\equiv 0\mod 3 \\ p_1=0 &: \frac{t}{h}\equiv 1\mod 3 \\
 p_2=0 &: \frac{t}{h}\equiv 2\mod 3 \end{cases} \]
for $t\in\T$; that is, $p$ has infinitely many zeros. Clearly $p$ is a 3-cycle, and 
$$ 0\ne |e_p(3h)|=|1+hp_0||1+hp_1||1+hp_2|=|1+hp_0|\ne 1 $$ 
holds. The following result is obtained by Theorem \ref{thm43} and simple calculations. Our main equation with this $p$ has Ulam stability with Ulam constant $K_0$ given by \eqref{nKmax0}. Note that
$$ S_0=\frac{3}{|1+hp_0|}, \qquad S_1=2+\frac{1}{|1+hp_0|}, \qquad S_2=1+\frac{2}{|1+hp_0|}. $$
In this case, the Ulam constant $K_0$ is 
\[ K_0 = 
 \begin{cases} 
  \frac{h|1+hp_0|S_1}{|1+hp_0|-1} &: 1<|1+hp_0|, \\
  \frac{h|1+hp_0|S_0}{1-|1+hp_0|} &: 0<|1+hp_0|<1.
 \end{cases}\]
Moreover, $\frac{h|1+hp_0|S_1}{|1+hp_0|-1}$ is the best (minimum) constant for Ulam stability when $|1+hp_0|>1$. See also Remarks \ref{remark2.3} and \ref{hilgercircle}. \hfill$\diamondsuit$
\end{remark}

\section*{Acknowledgements}
The second author was supported by JSPS KAKENHI Grant Number JP17K14226.


\end{document}